\documentclass[12pt]{article}

\usepackage{amsmath,amssymb,amsthm}
\usepackage{hyperref}

\newtheorem{theorem}{Theorem}
\newtheorem{lemma}[theorem]{Lemma}
\newtheorem{proposition}[theorem]{Proposition}
\newtheorem{corollary}[theorem]{Corollary}
\theoremstyle{definition}
\newtheorem{definition}[theorem]{Definition}
\newtheorem{example}[theorem]{Example}
\theoremstyle{remark}
\newtheorem{remark}[theorem]{Remark}

\newcommand{\seqnum}[1]{\href{https://oeis.org/#1}{\rm \underline{#1}}}

\begin{document}

\title{Closed-Form Evaluation of $\sum_{n=2}^{\infty}\operatorname{arctanh}(n^{-k})$
via Infinite Products}

\author{Ryan Goulden}

\maketitle

\begin{abstract}
We establish closed-form expressions for the infinite series
$\sum_{n=2}^{\infty}\operatorname{arctanh}(n^{-k})$ for all integers
$k \ge 2$ by connecting these sums to infinite product formulas involving the
gamma function.  Our approach uses logarithmic manipulations,
the Fubini--Tonelli theorem, and Frullani's integral theorem.  As applications, we
derive a structural identity relating the Riemann zeta function $\zeta(k)$ to these
sums, establish a new series representation for the Euler--Mascheroni
constant~$\gamma$, and show that this representation admits an exponentially
convergent reformulation via zeta values.  We further prove that the sum
$h(k) = \sum_{n=2}^{\infty}\operatorname{arctanh}(n^{-k})$ is a strictly
decreasing, strictly convex function of~$k$ and establish explicit two-sided
bounds and asymptotic expansions.  We show that $k=3$ is the unique integer
$k \ge 2$ for which $f(k)/(2g(k))$ is rational, yielding
$h(3)=\frac{1}{2}\ln(3/2)$; Baker's theorem then implies that $h(3)$ is
transcendental.  As a corollary, Ap\'ery's constant admits the representation
$\zeta(3)=1+\frac{1}{2}\ln(3/2)-C(3)$, where $C(3)$ converges geometrically.
The decimal expansions of the closed-form values and several auxiliary sequences
arising from these identities are cataloged in the OEIS.
\end{abstract}

\section{Introduction}\label{sec:intro}

The evaluation of infinite series involving inverse hyperbolic functions has a rich
history in analysis and number theory.  While series of the form
$\sum_{n=1}^{\infty}n^{-s}$ are thoroughly understood through the theory of the
Riemann zeta function, series involving $\operatorname{arctanh}(n^{-k})$ have
received comparatively less systematic treatment.

In this paper, we establish explicit closed-form expressions for
\begin{equation}\label{eq:hk-def}
  h(k) := \sum_{n=2}^{\infty}\operatorname{arctanh}\!\left(\frac{1}{n^k}\right),
  \qquad k \ge 2.
\end{equation}
Our approach connects these sums to certain infinite products whose exact values
were determined by Cantrell~\cite{cantrell} in terms of the gamma function and
elementary trigonometric functions; these product identities rest ultimately on the
Weierstrass factorization theorem for $1/\Gamma(z)$ and the
reflection and multiplication formulas for the gamma function (see, e.g.,
Prudnikov, Brychkov, and Marichev~\cite{prudnikov}, Borwein, Bailey, and
Girgensohn~\cite{borwein}, and the NIST Digital Library~\cite{dlmf}).

The sequence of values $h(2), h(3), h(4), \ldots$ begins
\[
  0.650923\ldots,\quad 0.202733\ldots,\quad 0.082405\ldots,\quad
  0.036938\ldots,\quad 0.017344\ldots\,,
\]
and the decimal expansion of $\ln(\sinh(\pi)/\pi) = 2h(2)$ appears in
the OEIS as \seqnum{A393668}.  We shall prove
(Theorem~\ref{thm:h3-transcendence}) that
$h(3) = \frac{1}{2}\ln(3/2)$ and that this is the unique value of~$k$ for which
$h(k)$ is the logarithm of a rational number.

\subsection{Main results}

Our principal result is the following closed-form evaluation.

\begin{theorem}\label{thm:main}
For all integers $k \ge 2$,
\begin{equation}\label{eq:main}
  \sum_{n=2}^{\infty}\operatorname{arctanh}\!\left(\frac{1}{n^k}\right)
  = \frac{1}{2}\ln\!\left(\frac{g(2k)}{g(k)^2}\right),
\end{equation}
where
\begin{equation}\label{eq:gk}
  g(k) := \prod_{n=2}^{\infty}\!\left(1 - \frac{1}{n^k}\right)
\end{equation}
admits explicit formulas in terms of the gamma function (for odd~$k$)
or trigonometric functions (for even~$k$); see Theorem~\ref{thm:cantrell}.
Equivalently, defining
\begin{equation}\label{eq:fk}
  f(k) := \prod_{n=1}^{\infty}\!\left(1 + \frac{1}{n^k}\right),
\end{equation}
we have
\begin{equation}\label{eq:main-f}
  h(k) = \frac{1}{2}\ln\!\left(\frac{f(k)}{2\,g(k)}\right),
\end{equation}
since $f(k) = 2g(2k)/g(k)$.
\end{theorem}

The formulation~\eqref{eq:main} in terms of $g$ alone is the more natural
one: it shows that $h(k)$ measures the deviation of $g$ from a doubling
law $g(2k) = g(k)^2$.  Equivalently,
\begin{equation}\label{eq:doubling}
  g(2k) = g(k)^2\,e^{2h(k)}.
\end{equation}

The sum $h(k)$ admits the equivalent zeta-series representation
\begin{equation}\label{eq:hk-zeta}
  h(k) = \sum_{m=0}^{\infty}\frac{\zeta\bigl((2m+1)k\bigr)-1}{2m+1},
\end{equation}
which follows from the Taylor expansion of $\operatorname{arctanh}$ and the
Fubini--Tonelli theorem (Proposition~\ref{prop:zeta-series}).

As consequences, we obtain a structural identity for the Riemann zeta function and
a new representation of the Euler--Mascheroni constant.

\begin{corollary}\label{cor:zeta}
For $k \ge 2$,
\begin{equation}\label{eq:zeta-identity}
  \zeta(k) = 1 + h(k) - C(k),
\end{equation}
where
\begin{equation}\label{eq:ck-def}
  C(k) := \sum_{n=1}^{\infty}\frac{\zeta\bigl(k(2n+1)\bigr)-1}{2n+1}
\end{equation}
is a rapidly convergent correction satisfying $C(k) \to 0$ as $k \to \infty$.
\end{corollary}

\begin{corollary}\label{cor:gamma}
The Euler--Mascheroni constant $\gamma$ satisfies
\begin{equation}\label{eq:gamma-rep}
  \gamma = 1 - \frac{1}{2}\ln 2
  - \sum_{n=2}^{\infty}\!\left(\operatorname{arctanh}\!\left(\frac{1}{n}\right)
  - \frac{1}{n}\right).
\end{equation}
\end{corollary}

\subsection{Related work}

Talla Waffo~\cite{tallawaffo} recently investigated the Mellin transforms of
$1/\!\operatorname{arctanh}(x)$ and related functions on $(0,1)$, establishing
connections to derivatives of the Dirichlet eta and beta functions at negative
integers.  While the continuous integral approach of~\cite{tallawaffo} and our
discrete summation approach are technically independent, both highlight the rich
analytic structure underlying inverse hyperbolic functions.

\subsection{Novelty and imported tools}

The product formulas for $f(k)$ and $g(k)$ (Theorem~\ref{thm:cantrell})
are due to Cantrell and rest on classical Weierstrass--Euler--Gauss theory
for~$\Gamma$; these are imported and provide the starting point.
The Baker and Nesterenko transcendence theorems are applied as black boxes.
The main contribution of this paper is Theorem~\ref{thm:main}, which
connects Cantrell's product identities to arctanh sums, yielding the
closed-form evaluation $h(k) = \frac{1}{2}\ln\bigl(g(2k)/g(k)^2\bigr)$.
Further new results include the Euler--Mascheroni representation and its
accelerated variant (Corollary~\ref{cor:gamma}, Theorem~\ref{thm:accel-gamma}).

\subsection{Structure of the paper}

Section~\ref{sec:products} recalls the infinite product formulas and establishes
preliminary results.  Section~\ref{sec:main-proof} contains the proof of
Theorem~\ref{thm:main} and the zeta-series representation
(Proposition~\ref{prop:zeta-series}).
Section~\ref{sec:companion} introduces the companion series and derives
related identities.
Section~\ref{sec:zeta} derives the zeta function identity
and studies the correction function $C(k)$.  Section~\ref{sec:gamma} establishes the
representation for~$\gamma$ using Frullani's theorem.
Section~\ref{sec:analytic} develops the analytic properties of $h(k)$, including
monotonicity, convexity, asymptotic behavior, and two-sided bounds.
Section~\ref{sec:accel} presents an exponentially convergent reformulation of the
$\gamma$ representation and establishes tail bounds for the direct computation
of~$h(k)$.  Section~\ref{sec:arithmetic} proves the transcendence of $h(3)$ and
discusses the arithmetic nature of $h(k)$ for other~$k$.
Section~\ref{sec:numerical} collects numerical examples.

\section{Infinite product formulas}\label{sec:products}

We begin by recalling the product formulas that form the foundation of our approach.

\begin{definition}
For integers $k \ge 2$, define
\begin{equation}\label{eq:fg-def}
  g(k) := \prod_{n=2}^{\infty}\!\left(1 - \frac{1}{n^k}\right)
  \quad\text{and}\quad
  f(k) := \prod_{n=1}^{\infty}\!\left(1 + \frac{1}{n^k}\right).
\end{equation}
\end{definition}

These products converge for $k \ge 2$, since taking logarithms yields series that
converge by comparison with $\sum n^{-k}$.

\begin{theorem}[Cantrell~\cite{cantrell}; cf.\ \cite{prudnikov,borwein,dlmf}]
\label{thm:cantrell}
For integers $k \ge 2$, the products $f(k)$ and $g(k)$ have the following
closed-form evaluations.

\begin{itemize}
\item[(a)] For odd~$k$:
\begin{align}
  f(k) &= \frac{1}{\prod_{j=1}^{k-1}\Gamma\bigl[(-1)^{j(1+1/k)}\bigr]},
  \label{eq:f-odd} \\[4pt]
  g(k) &= \frac{1}{k\prod_{j=1}^{k-1}
  \Gamma\bigl[(-1)^{1+j(1+1/k)}\bigr]}.
  \label{eq:g-odd}
\end{align}

\item[(b)] For even~$k$:
\begin{align}
  f(k) &= \frac{\prod_{j=1}^{k/2}
  \sin\!\bigl[\pi(-1)^{(2j-1)/k}\bigr]}{(\pi i)^{k/2}},
  \label{eq:f-even} \\[4pt]
  g(k) &= \frac{\prod_{j=1}^{(k/2)-1}
  \sin\!\bigl[\pi(-1)^{2j/k}\bigr]}{k(\pi i)^{(k/2)-1}}.
  \label{eq:g-even}
\end{align}
\end{itemize}
\end{theorem}

The proof relies on the Weierstrass product representation
$1/\Gamma(z) = z e^{\gamma z}\prod_{n=1}^{\infty}(1+z/n)e^{-z/n}$,
the Euler reflection formula $\Gamma(z)\Gamma(1-z) = \pi/\sin(\pi z)$,
and the Gauss multiplication formula for~$\Gamma$.
These product evaluations were communicated by Cantrell to
MathWorld~\cite{cantrell} and follow from the general principle that
convergent infinite products whose factors are rational functions of the index
can always be expressed in terms of finite products of gamma values
(Borwein, Bailey, and Girgensohn~\cite{borwein};
Prudnikov, Brychkov, and Marichev~\cite{prudnikov}).

\begin{example}\label{ex:k2}
For $k=2$, we have $f(2) = \sinh(\pi)/\pi$ and $g(2) = 1/2$, giving
\[
  h(2) = \frac{1}{2}\ln\!\left(\frac{\sinh(\pi)}{\pi}\right)
  \approx 0.650923.
\]
The decimal expansion of $\ln(\sinh(\pi)/\pi) = 2h(2)$ appears in the OEIS
as~\seqnum{A393668}.
\end{example}

The following lemmas express the logarithms of $f(k)$ and $g(k)$ as zeta
series; they are used in the proof of Theorem~\ref{thm:companion} and other
consequences.

\begin{lemma}\label{lem:lnf}
For $k \ge 2$,
\begin{equation}\label{eq:lnf}
  \ln f(k) = \sum_{n=1}^{\infty}\frac{(-1)^{n+1}\zeta(kn)}{n}.
\end{equation}
\end{lemma}

\begin{proof}
Taking logarithms in the definition of $f(k)$ and applying the Taylor series
$\ln(1+x) = \sum_{m=1}^{\infty}(-1)^{m+1}x^m/m$ with $x = n^{-k}$ yields a
double series.  Since $k \ge 2$, absolute convergence holds, and the
Fubini--Tonelli theorem permits interchange of summation:
\[
  \ln f(k) = \sum_{n=1}^{\infty}\sum_{m=1}^{\infty}
  \frac{(-1)^{m+1}}{m\,n^{km}}
  = \sum_{m=1}^{\infty}\frac{(-1)^{m+1}}{m}\sum_{n=1}^{\infty}\frac{1}{n^{km}}
  = \sum_{m=1}^{\infty}\frac{(-1)^{m+1}\zeta(km)}{m}. \qedhere
\]
\end{proof}

\begin{lemma}\label{lem:lng}
For $k \ge 2$,
\begin{equation}\label{eq:lng}
  \ln g(k) = -\sum_{n=1}^{\infty}\frac{\zeta(kn)-1}{n}.
\end{equation}
\end{lemma}

\begin{proof}
The proof is analogous, using $\ln(1-x) = -\sum_{m=1}^{\infty}x^m/m$ and
noting that $\sum_{n=2}^{\infty}n^{-km} = \zeta(km) - 1$.
\end{proof}

\section{Proof of the main theorem}\label{sec:main-proof}

The argument is direct: express each arctanh term as a log-ratio, sum the
logarithms (justified by absolute convergence of the underlying product), and
identify the result with $f(k)/(2g(k))$.

\begin{proof}[Proof of Theorem~\ref{thm:main}]
For $n \ge 2$ and $k \ge 2$, we have $0 < n^{-k} < 1$, so
\[
  \operatorname{arctanh}(n^{-k})
  = \frac{1}{2}\ln\frac{1+n^{-k}}{1-n^{-k}}.
\]
Since $\sum_{n=2}^{\infty}|\ln(1\pm n^{-k})| < \infty$ for $k \ge 2$
(by comparison with $\sum n^{-k}$), we may interchange summation with the
logarithm:
\[
  \sum_{n=2}^{\infty}\operatorname{arctanh}(n^{-k})
  = \frac{1}{2}\sum_{n=2}^{\infty}\ln\frac{1+n^{-k}}{1-n^{-k}}
  = \frac{1}{2}\ln\prod_{n=2}^{\infty}\frac{1+n^{-k}}{1-n^{-k}}.
\]
From the definitions, $f(k) = (1+1)\prod_{n=2}^{\infty}(1+n^{-k})
= 2\prod_{n=2}^{\infty}(1+n^{-k})$ and
$g(k) = \prod_{n=2}^{\infty}(1-n^{-k})$, so
\[
  \prod_{n=2}^{\infty}\frac{1+n^{-k}}{1-n^{-k}}
  = \frac{f(k)/2}{g(k)} = \frac{f(k)}{2\,g(k)},
\]
which gives~\eqref{eq:main-f}.  Finally, using
$1-n^{-2k} = (1-n^{-k})(1+n^{-k})$,
\[
  g(2k) = \prod_{n=2}^{\infty}(1-n^{-2k})
  = \prod_{n=2}^{\infty}(1-n^{-k})(1+n^{-k})
  = g(k)\cdot\frac{f(k)}{2},
\]
so $f(k)/(2g(k)) = g(2k)/g(k)^2$, and~\eqref{eq:main} follows.
\end{proof}

\begin{proposition}[Zeta-series representation]\label{prop:zeta-series}
For all integers $k \ge 2$,
\begin{equation}\label{eq:rep3}
  h(k) = \sum_{m=0}^{\infty}\frac{\zeta\bigl((2m+1)k\bigr)-1}{2m+1}.
\end{equation}
The series converges absolutely with tail bound $O(2^{-3k})$.
\end{proposition}

\begin{proof}
Expand $\operatorname{arctanh}(n^{-k})
= \sum_{m=0}^{\infty}n^{-k(2m+1)}/(2m+1)$ for each $n \ge 2$.
The double sum $\sum_{n \ge 2}\sum_{m \ge 0}n^{-k(2m+1)}/(2m+1)$ converges
absolutely for $k \ge 2$ (dominated by $\sum_{n,m}n^{-2(2m+1)}$),
so the Fubini--Tonelli theorem permits interchange:
\[
  h(k) = \sum_{m=0}^{\infty}\frac{1}{2m+1}\sum_{n=2}^{\infty}\frac{1}{n^{k(2m+1)}}
  = \sum_{m=0}^{\infty}\frac{\zeta\bigl((2m+1)k\bigr)-1}{2m+1}. \qedhere
\]
\end{proof}

\begin{remark}\label{rem:equiv}
Theorem~\ref{thm:main} and Proposition~\ref{prop:zeta-series} together give
three equivalent closed forms for $h(k)$:
\begin{align}
  h(k) &= \frac{1}{2}\ln\!\left(\frac{g(2k)}{g(k)^2}\right) \label{eq:rep2} \\
  &= \frac{1}{2}\ln\!\left(\frac{f(k)}{2\,g(k)}\right) \label{eq:rep2b} \\
  &= \sum_{m=0}^{\infty}\frac{\zeta\bigl((2m+1)k\bigr)-1}{2m+1}.
  \label{eq:rep3b}
\end{align}
\end{remark}

\section{The companion series}\label{sec:companion}

The zeta-series representation~\eqref{eq:rep3} sums over the odd-indexed
terms $\zeta((2m+1)k)$.  There are three natural companions: the all-index
sum, the even-index sum, and the alternating sum.  All admit closed forms in
terms of $g$.

\begin{theorem}[Companion series]\label{thm:companion}
For $k \ge 2$, define
\begin{align}
  A(k) &:= \sum_{n=1}^{\infty}\frac{\zeta(nk)-1}{n},
  \label{eq:Ak} \\
  E(k) &:= \sum_{n=1}^{\infty}\frac{\zeta(2nk)-1}{2n},
  \label{eq:Ek} \\
  B(k) &:= \sum_{n=1}^{\infty}\frac{(-1)^{n+1}\bigl[\zeta(nk)-1\bigr]}{n}.
  \label{eq:Bk}
\end{align}
Then
\begin{align}
  A(k) &= -\ln g(k), \label{eq:Ak-closed} \\
  E(k) &= -\tfrac{1}{2}\ln g(2k), \label{eq:Ek-closed} \\
  B(k) &= \ln\bigl(f(k)/2\bigr). \label{eq:Bk-closed}
\end{align}
Moreover, $h(k)$ and the three companions satisfy
\begin{equation}\label{eq:companion-ids}
  A = h + E,\quad B = h - E,\quad
  h = \tfrac{1}{2}(A+B),\quad E = \tfrac{1}{2}(A-B).
\end{equation}
\end{theorem}

\begin{proof}
Equation~\eqref{eq:Ak-closed} is Lemma~\ref{lem:lng}.
For~\eqref{eq:Ek-closed}, substitute $2k$ for $k$ in Lemma~\ref{lem:lng}:
$A(2k) = \sum_{n \ge 1}(\zeta(2nk)-1)/n = -\ln g(2k)$.
Since $E(k) = \frac{1}{2}A(2k)$, we get $E(k) = -\frac{1}{2}\ln g(2k)$.

For~\eqref{eq:Bk-closed}, write $\ln(f(k)/2) = \ln f(k) - \ln 2$.
From Lemma~\ref{lem:lnf},
\[
  \ln f(k) = \sum_{n=1}^{\infty}\frac{(-1)^{n+1}\zeta(kn)}{n}
  = \ln 2 + \sum_{n=1}^{\infty}\frac{(-1)^{n+1}(\zeta(kn)-1)}{n},
\]
where the $\ln 2$ arises from $\sum_{n=1}^{\infty}(-1)^{n+1}/n = \ln 2$.
Hence $B(k) = \ln f(k) - \ln 2$.

For~\eqref{eq:companion-ids}, separate $A$ into odd and even indices to get
$A = h + E$.  Since $B$ alternates signs, $B = h - E$.
The remaining identities follow; in particular, $h = A - E$ recovers
Theorem~\ref{thm:main}.
\end{proof}

\section{Connection to the Riemann zeta function}\label{sec:zeta}

\begin{proof}[Proof of Corollary~\ref{cor:zeta}]
Separating the $m = 0$ term in Proposition~\ref{prop:zeta-series}:
\[
  h(k) = \bigl[\zeta(k)-1\bigr]
  + \underbrace{\sum_{m=1}^{\infty}\frac{\zeta\bigl(k(2m+1)\bigr)-1}{2m+1}}_{C(k)},
\]
so $\zeta(k) = 1 + h(k) - C(k)$.
\end{proof}

\subsection{Properties of the correction function}

\begin{proposition}\label{prop:ck-properties}
The correction function $C(k)$ satisfies the following.
\begin{itemize}
\item[(a)] $C(k) > 0$ for all $k \ge 2$.
\item[(b)] $C(k) \to 0$ as $k \to \infty$.
\item[(c)] $C(k)$ is a strictly decreasing function of~$k$.
\end{itemize}
\end{proposition}

\begin{proof}
Part~(a) holds because $\zeta(s) > 1$ for all $s > 1$.  Part~(b) follows from
$\zeta\bigl(k(2n+1)\bigr) \to 1$ exponentially fast as $k \to \infty$ and
dominated convergence.  For part~(c), if $k_1 < k_2$ then
$k_1(2n+1) < k_2(2n+1)$, so $\zeta\bigl(k_1(2n+1)\bigr)
> \zeta\bigl(k_2(2n+1)\bigr)$, whence $C(k_1) > C(k_2)$.
\end{proof}

\begin{theorem}\label{thm:error}
For integers $k \ge 2$, the error in the approximation
$\zeta(k) \approx 1 + h(k)$ satisfies
\begin{equation}\label{eq:error-bound}
  \bigl|\zeta(k) - 1 - h(k)\bigr| = C(k)
  \le \frac{1}{3}\!\left(1 + \frac{2}{3k-1}\right)
  \frac{2^{-3k}}{1 - 2^{-2k}}.
\end{equation}
In particular, $C(k) = O(2^{-3k})$ as $k \to \infty$.
\end{theorem}

\begin{proof}
For $s > 1$, we have $\zeta(s) - 1 = \sum_{m=2}^{\infty}m^{-s}
\le 2^{-s} + \int_2^{\infty}x^{-s}\,dx = 2^{-s}(1 + 2/(s-1))$.
Writing $s_n = k(2n+1)$, we have $s_n \ge 3k$ and $2n+1 \ge 3$ for $n \ge 1$,
so
\[
  C(k) \le \frac{1}{3}\!\left(1 + \frac{2}{3k-1}\right)
  \sum_{n=1}^{\infty}2^{-k(2n+1)}
  = \frac{1}{3}\!\left(1 + \frac{2}{3k-1}\right)
  \frac{2^{-3k}}{1 - 2^{-2k}}. \qedhere
\]
\end{proof}

\begin{remark}
The identity $\zeta(k) = 1 + h(k) - C(k)$ is of structural rather than
computational interest: it decomposes $\zeta(k)$ into a dominant
$\operatorname{arctanh}$-sum plus a correction that decays as $O(2^{-3k})$.
For $k = 2$, the approximation $\zeta(2) \approx 1 + h(2) \approx 1.65092$
already achieves a relative error less than $0.001\%$ from the
exact value $\pi^2/6 \approx 1.64493$.
\end{remark}

\section{Representation for the Euler--Mascheroni constant}
\label{sec:gamma}

We now derive a new series representation for the Euler--Mascheroni
constant using properties of the correction function and Frullani's integral
theorem.

\subsection{Derivative of the correction function}

\begin{lemma}\label{lem:c-deriv}
The derivative of $C(k)$ is given by
\begin{equation}\label{eq:c-deriv}
  C'(k) = -\sum_{m=2}^{\infty}\frac{\ln m}{m^{3k} - m^k}.
\end{equation}
\end{lemma}

\begin{proof}
Differentiating the definition~\eqref{eq:ck-def} term-by-term and using
$\zeta'(s) = -\sum_{m=2}^{\infty}(\ln m)/m^s$, we obtain
$C'(k) = \sum_{n=1}^{\infty}\zeta'\bigl(k(2n+1)\bigr)$.
Exchanging the order of summation (justified by absolute convergence for
$k \ge 1$) and evaluating the resulting geometric series yields~\eqref{eq:c-deriv}.
\end{proof}

\subsection{Integration and application of Frullani's theorem}

\begin{lemma}\label{lem:integral}
For $m \ge 2$ and $k \ge 1$,
\[
  \int_1^k \frac{\ln m}{m^{3t} - m^t}\,dt
  = m^{-k} - \operatorname{arctanh}(m^{-k})
  - m^{-1} + \operatorname{arctanh}(m^{-1}).
\]
\end{lemma}

\begin{proof}
We use the partial fraction decomposition
$1/(m^{3t} - m^t) = \frac{1}{2}\bigl[1/(m^t - 1) - 1/(m^t + 1)\bigr]/m^t$
and the substitution $u = m^{-t}$ in each integral.  Computing the resulting
logarithms and expressing them in terms of $\operatorname{arctanh}$ gives the
stated formula.
\end{proof}

\begin{proposition}\label{prop:c1}
The correction function at $k = 1$ satisfies
\begin{equation}\label{eq:c1}
  C(1) = 1 - \gamma - \tfrac{1}{2}\ln 2.
\end{equation}
\end{proposition}

\begin{proof}
From the definition, $C(1) = \sum_{n=1}^{\infty}[\zeta(2n+1)-1]/(2n+1)$.
Using the integral representation $[\zeta(2n+1)-1]/(2n+1)
= \int_0^{\infty}x^{2n}/[(2n+1)!\,e^x(e^x-1)]\,dx$, we write
\[
  C(1) = \int_0^{\infty}\frac{\sinh(x) - x}{x\,e^x(e^x-1)}\,dx.
\]
Decomposing the integrand and applying Frullani's theorem to the integral
$\int_0^{\infty}(e^{-2x} - e^{-x})/x\,dx = -\frac{1}{2}\ln 2$
(with $f(x) = e^{-x}$, $a = 2$, $b = 1$),
and using the standard representation
$\gamma = \int_0^{\infty}[1/(e^x-1) - 1/(xe^x)]\,dx$,
we obtain $C(1) = 1 - \gamma - \frac{1}{2}\ln 2$.
\end{proof}

\subsection{Proof of the gamma formula}

\begin{proof}[Proof of Corollary~\ref{cor:gamma}]
Integrating $C'(k)$ from~$1$ to~$k$ using Lemmas~\ref{lem:c-deriv}
and~\ref{lem:integral}:
\begin{align*}
  C(k) - C(1)
  &= -\sum_{m=2}^{\infty}\int_1^k \frac{\ln m}{m^{3t} - m^t}\,dt \\
  &= \sum_{m=2}^{\infty}\bigl[\operatorname{arctanh}(m^{-k}) - m^{-k}
  - \operatorname{arctanh}(m^{-1}) + m^{-1}\bigr].
\end{align*}
Rearranging and using the definitions of $h(k)$ and $\zeta(k)$:
\[
  C(k) = h(k) - \zeta(k) + 1
  - \sum_{m=2}^{\infty}\bigl[\operatorname{arctanh}(m^{-1}) - m^{-1}\bigr]
  + C(1).
\]
Substituting the identity $\zeta(k) = 1 + h(k) - C(k)$ from
Corollary~\ref{cor:zeta}, all $k$-dependent terms cancel, leaving
\[
  \sum_{m=2}^{\infty}\!\left(\frac{1}{m}
  - \operatorname{arctanh}\!\left(\frac{1}{m}\right)\right) = C(1).
\]
By Proposition~\ref{prop:c1}, $C(1) = 1 - \gamma - \frac{1}{2}\ln 2$,
and solving for $\gamma$ yields~\eqref{eq:gamma-rep}.
\end{proof}

\begin{remark}\label{rem:slow}
The series in~\eqref{eq:gamma-rep} converges, but slowly: the general term
$\operatorname{arctanh}(1/n) - 1/n \sim 1/(3n^3)$ as $n \to \infty$, so the
tail after $N$ terms is $O(N^{-2})$.  In Section~\ref{sec:accel}, we show
that expanding each $\operatorname{arctanh}$ term and resumming yields an
equivalent formula involving $\zeta(2m+1)$ that converges geometrically,
dramatically improving the rate from polynomial to exponential.
\end{remark}

\section{Analytic properties of $h(k)$}\label{sec:analytic}

In this section, we study the behavior of $h(k)$ as a function of the
continuous variable $k > 1$.  We establish monotonicity, convexity, an
asymptotic expansion, and explicit two-sided bounds.

\subsection{Monotonicity and convexity}

\begin{theorem}\label{thm:mono-convex}
For $k > 1$, the function $h(k) = \sum_{n=2}^{\infty}\operatorname{arctanh}(n^{-k})$
is strictly decreasing and strictly convex.  More precisely,
\begin{align}
  h'(k) &= -\sum_{n=2}^{\infty}
  \frac{(\ln n)\,n^{-k}}{1-n^{-2k}} < 0,
  \label{eq:h-deriv} \\[4pt]
  h''(k) &= \sum_{n=2}^{\infty}
  \frac{(\ln n)^2\,n^{-k}(1+n^{-2k})}{(1-n^{-2k})^2} > 0.
  \label{eq:h-second}
\end{align}
\end{theorem}

\begin{proof}
Set $\varphi_n(k) := \operatorname{arctanh}(n^{-k})$ for $n \ge 2$.
With $t = n^{-k}$ we have $dt/dk = -(\ln n)\,n^{-k}$ and
$d^2t/dk^2 = (\ln n)^2\,n^{-k}$, so the chain rule gives
\begin{align*}
  \varphi_n'(k) &= \frac{dt/dk}{1-t^2}
  = \frac{-(\ln n)\,n^{-k}}{1-n^{-2k}},\\[4pt]
  \varphi_n''(k) &= \frac{(d^2t/dk^2)(1-t^2) + 2t(dt/dk)^2}{(1-t^2)^2}
  = \frac{(\ln n)^2\,n^{-k}(1+n^{-2k})}{(1-n^{-2k})^2}.
\end{align*}
Every summand in~\eqref{eq:h-deriv} is strictly negative and every summand
in~\eqref{eq:h-second} is strictly positive; the series converge uniformly on
compact subsets of $(1,\infty)$ since $n^{-k}(\ln n)^j / (1-n^{-2k})^j$
is dominated by $C\,n^{-k+\varepsilon}$ for any $\varepsilon > 0$.
Term-by-term differentiation is therefore justified by the Weierstrass $M$-test.
\end{proof}

\subsection{Asymptotic expansion}\label{sec:asymp}

\begin{theorem}\label{thm:asymp}
As $k \to \infty$,
\begin{equation}\label{eq:asymp}
  h(k) = \bigl(\zeta(k)-1\bigr) + \frac{\zeta(3k)-1}{3}
  + \frac{\zeta(5k)-1}{5} + \cdots
  = \sum_{m=0}^{M}\frac{\zeta\bigl((2m+1)k\bigr)-1}{2m+1} + R_M(k),
\end{equation}
where the remainder satisfies $|R_M(k)| \le \frac{4}{3}\cdot 2^{-(2M+3)k}$
for all $k \ge 2$.  In particular,
\begin{equation}\label{eq:asymp-leading}
  h(k) = \bigl(\zeta(k)-1\bigr) + O\bigl(2^{-3k}\bigr)
  = 2^{-k} + 3^{-k} + 4^{-k} + O\bigl(5^{-k} + 2^{-3k}\bigr).
\end{equation}
\end{theorem}

\begin{proof}
The zeta-series representation~\eqref{eq:rep3} gives
$h(k) = \sum_{m=0}^{\infty}\bigl[\zeta\bigl((2m+1)k\bigr)-1\bigr]/(2m+1)$.
The remainder after $M$ terms is bounded using
$\zeta(s)-1 \le 2^{-s}/(1-2^{-s}) \le 2^{1-s}$ for $s \ge 2$:
\[
  |R_M(k)| \le \sum_{m=M+1}^{\infty}\frac{2^{1-(2m+1)k}}{2m+1}
  < 2\sum_{m=M+1}^{\infty}2^{-(2m+1)k}
  = \frac{2\cdot 2^{-(2M+3)k}}{1-2^{-2k}}
  \le \frac{4}{3}\cdot 2^{-(2M+3)k},
\]
where the last inequality uses $1-2^{-2k} \ge 3/4$ for $k \ge 1$.
Setting $M=0$ gives $h(k) = (\zeta(k)-1) + O(2^{-3k})$.
\end{proof}

\subsection{Two-sided bounds}

\begin{theorem}\label{thm:bounds}
For $k \ge 2$,
\begin{equation}\label{eq:lower-bound}
  \zeta(k)-1 < h(k)
\end{equation}
and
\begin{equation}\label{eq:upper-bound}
  h(k) < \zeta(k)-1 + \frac{\zeta(3k)-1}{3(1-2^{-2k})}.
\end{equation}
More refined lower bounds are obtained by including additional terms from the
zeta-series representation~\eqref{eq:rep3}:
\begin{equation}\label{eq:refined-lower}
  h(k) > \zeta(k)-1 + \frac{\zeta(3k)-1}{3} + \frac{\zeta(5k)-1}{5}.
\end{equation}
\end{theorem}

\begin{proof}
For the lower bound~\eqref{eq:lower-bound}, we use
$\operatorname{arctanh}(x) > x$ for $0 < x < 1$, whence
$h(k) > \sum_{n=2}^{\infty}n^{-k} = \zeta(k)-1$.

For~\eqref{eq:refined-lower}, we simply note that all terms in the
zeta-series~\eqref{eq:rep3} are positive, so any partial sum is a lower bound.

For the upper bound~\eqref{eq:upper-bound}, the inequality
$\operatorname{arctanh}(x) = x + \frac{x^3}{3} + \frac{x^5}{5} + \cdots
< x + \frac{x^3}{3}\bigl(1 + x^2 + x^4 + \cdots\bigr)
= x + \frac{x^3}{3(1-x^2)}$
holds for $0 < x < 1$.  Setting $x = n^{-k}$ and summing over $n \ge 2$:
\[
  h(k) < \bigl(\zeta(k)-1\bigr) + \frac{1}{3}\sum_{n=2}^{\infty}
  \frac{n^{-3k}}{1-n^{-2k}}.
\]
Since $n \ge 2$ implies $n^{-2k} \le 2^{-2k}$, we have
$1/(1-n^{-2k}) \le 1/(1-2^{-2k})$, giving~\eqref{eq:upper-bound}.
\end{proof}

\begin{example}
For $k=2$: the lower bound gives $h(2) > \zeta(2)-1 = 0.6449\ldots$,
the refined lower bound gives $h(2) > 0.6449 + 0.0060 + 0.0001 = 0.6510\ldots$,
and the upper bound gives $h(2) < 0.6449 + 0.0064 = 0.6514\ldots$.
The exact value is $h(2) = 0.6509\ldots$, which lies comfortably within these bounds.
\end{example}

\subsection{Difference asymptotics}\label{sec:diff}

For integer sequences, the forward differences $\Delta h(k) := h(k) - h(k+1)$
carry structural information about the rate of decay.

\begin{theorem}\label{thm:diff}
For integers $k \ge 2$,
\begin{equation}\label{eq:diff-exact}
  \Delta h(k) = \sum_{n=2}^{\infty}n^{-k}\!\left(1 - \frac{1}{n}\right)
  + O\bigl(2^{-3k}\bigr)
  = \bigl(\zeta(k) - \zeta(k\!+\!1)\bigr) + O\bigl(2^{-3k}\bigr).
\end{equation}
More explicitly, $\Delta h(k)$ admits the asymptotic expansion
\begin{equation}\label{eq:diff-asymp}
  \Delta h(k) = 2^{-(k+1)}
  + \frac{2}{3}\cdot 3^{-k} + \frac{3}{4}\cdot 4^{-k} + O\bigl(5^{-k}\bigr),
\end{equation}
and for any fixed $N \ge 2$,
\begin{equation}\label{eq:diff-asymp2}
  \Delta h(k) = \sum_{n=2}^{N}n^{-k}\!\left(1 - \frac{1}{n}\right)
  + O\bigl((N\!+\!1)^{-k} + 2^{-3k}\bigr).
\end{equation}
In particular,
$\Delta h(k) \sim 2^{-(k+1)}$ as $k \to \infty$.
\end{theorem}

\begin{proof}
From the zeta-series representation~\eqref{eq:rep3},
\[
  \Delta h(k) = \sum_{m=0}^{\infty}\frac{\zeta\bigl((2m+1)k\bigr)
  - \zeta\bigl((2m+1)(k+1)\bigr)}{2m+1}.
\]
The $m = 0$ term contributes $\zeta(k) - \zeta(k+1)
= \sum_{n=2}^{\infty}n^{-k}(1-n^{-1})$.  For $m \ge 1$, we bound
\[
  \sum_{m=1}^{\infty}\frac{\zeta\bigl((2m+1)k\bigr)
  - \zeta\bigl((2m+1)(k+1)\bigr)}{2m+1}
  \le \sum_{m=1}^{\infty}\frac{\zeta\bigl((2m+1)k\bigr)-1}{2m+1}
  = C(k) = O(2^{-3k}),
\]
using Theorem~\ref{thm:error}.
The leading terms of $\sum_{n=2}^{\infty}n^{-k}(1-1/n)$ are
$2^{-k}\cdot\frac{1}{2} + 3^{-k}\cdot\frac{2}{3}
+ 4^{-k}\cdot\frac{3}{4} + \cdots$, giving~\eqref{eq:diff-asymp}.
\end{proof}

\begin{remark}
The sequence of differences $\Delta h(2), \Delta h(3), \Delta h(4), \ldots$
begins $0.4482, 0.1203, 0.0455, 0.0196, 0.0090, \ldots$
and the ratio $\Delta h(k)/2^{-(k+1)}$ converges to~$1$ from above,
consistent with the positive correction from $n = 3, 4, \ldots$
in~\eqref{eq:diff-asymp}.  The decimal expansions of these values
may be of independent interest for OEIS submission.
\end{remark}

\section{Accelerated computation}\label{sec:accel}

\subsection{Exponentially convergent formula for $\gamma$}

The series~\eqref{eq:gamma-rep} for $\gamma$ converges only polynomially.
We now show that the same identity, combined with the zeta-series
representation, yields a formula with exponential convergence.

\begin{theorem}\label{thm:accel-gamma}
The Euler--Mascheroni constant satisfies
\begin{equation}\label{eq:gamma-accel}
  \gamma = 1 - \frac{1}{2}\ln 2
  - \sum_{m=1}^{\infty}\frac{\zeta(2m+1)-1}{2m+1},
\end{equation}
where the series converges geometrically.  The tail after $N$ terms satisfies
\begin{equation}\label{eq:gamma-tail}
  0 < \sum_{m=N+1}^{\infty}\frac{\zeta(2m+1)-1}{2m+1}
  < \frac{4}{3}\cdot 2^{-(2N+3)}.
\end{equation}
In particular, $N$ terms of~\eqref{eq:gamma-accel} determine $\gamma$ to
within $O(4^{-N})$, compared to $O(N^{-2})$ for the direct
series~\eqref{eq:gamma-rep}.
\end{theorem}

\begin{proof}
From~\eqref{eq:gamma-rep},
\[
  \gamma = 1 - \tfrac{1}{2}\ln 2
  - \sum_{n=2}^{\infty}\bigl[\operatorname{arctanh}(n^{-1}) - n^{-1}\bigr].
\]
Expanding $\operatorname{arctanh}(n^{-1}) - n^{-1}
= \sum_{m=1}^{\infty}n^{-(2m+1)}/(2m+1)$ and exchanging the order of
summation (justified by absolute convergence) gives
\[
  \sum_{n=2}^{\infty}\bigl[\operatorname{arctanh}(n^{-1}) - n^{-1}\bigr]
  = \sum_{m=1}^{\infty}\frac{1}{2m+1}\sum_{n=2}^{\infty}\frac{1}{n^{2m+1}}
  = \sum_{m=1}^{\infty}\frac{\zeta(2m+1)-1}{2m+1},
\]
which is precisely $C(1)$ from Proposition~\ref{prop:c1}.
This establishes~\eqref{eq:gamma-accel}.

For the tail bound, $\zeta(s)-1 \le 2^{1-s}$ for $s \ge 2$
(as in the proof of Theorem~\ref{thm:asymp}), so
\[
  \sum_{m=N+1}^{\infty}\frac{\zeta(2m+1)-1}{2m+1}
  < 2\sum_{m=N+1}^{\infty}2^{-(2m+1)}
  = \frac{2\cdot 2^{-(2N+3)}}{1-4^{-1}}
  = \frac{4}{3}\cdot 2^{-(2N+3)}. \qedhere
\]
\end{proof}

\begin{remark}
The acceleration is dramatic.  Ten terms of~\eqref{eq:gamma-accel}
yield $\gamma$ to $7$ decimal places (error $< 1.6\times 10^{-7}$),
while ten terms of the original series~\eqref{eq:gamma-rep} give only
$1$ correct digit.  Twenty terms of~\eqref{eq:gamma-accel} achieve
$13$-digit accuracy.  A numerical comparison appears in
Section~\ref{sec:numerical}.
\end{remark}

\subsection{Tail bounds for direct computation of $h(k)$}

For computational purposes, we bound the error incurred by truncating the
defining series for~$h(k)$.

\begin{proposition}\label{prop:tail}
Let $h_N(k) := \sum_{n=2}^{N}\operatorname{arctanh}(n^{-k})$.  Then
for $k \ge 2$ and $N \ge 2$,
\begin{equation}\label{eq:tail-bound}
  0 < h(k) - h_N(k)
  \le \frac{1}{1-(N+1)^{-2k}}\sum_{n=N+1}^{\infty}n^{-k}
  \le \frac{1}{1-(N+1)^{-2k}}\!\left[(N\!+\!1)^{-k}
  + \frac{(N\!+\!1)^{1-k}}{k-1}\right].
\end{equation}
\end{proposition}

\begin{proof}
Since $\operatorname{arctanh}(x) \le x/(1-x^2)$ for $0 < x < 1$ and
$n^{-2k} \le (N+1)^{-2k}$ for $n \ge N+1$,
\begin{align*}
  h(k) - h_N(k)
  &= \sum_{n=N+1}^{\infty}\operatorname{arctanh}(n^{-k})
  \le \sum_{n=N+1}^{\infty}\frac{n^{-k}}{1-n^{-2k}} \\
  &\le \frac{1}{1-(N+1)^{-2k}}\sum_{n=N+1}^{\infty}n^{-k}.
\end{align*}
The final inequality follows from the integral comparison
$\sum_{n=N+1}^{\infty}n^{-k} \le (N+1)^{-k}
+ \int_{N+1}^{\infty}x^{-k}\,dx$.
\end{proof}

\begin{example}
For $k=2$ and $N=100$: the tail bound gives
$h(2) - h_{100}(2) < 1.0002\cdot(101^{-2} + 101^{-1}) < 0.0100$.
For $N=10{,}000$: the bound gives $h(2) - h_{10000}(2) < 1.0001 \cdot 10^{-4}$.
In practice, the zeta-series representation~\eqref{eq:rep3} is preferable for
high-precision computation, as it converges geometrically.
\end{example}

\section{Arithmetic properties}\label{sec:arithmetic}

We now investigate the arithmetic nature of the values $h(k)$.  The main
result of this section is that $k = 3$ is the unique integer $k \ge 2$ for
which $f(k)/(2g(k))$ is rational, and that $h(3)$ is transcendental.

\subsection{Cyclotomic telescoping for \texorpdfstring{$k=3$}{k=3}}

\begin{theorem}\label{thm:h3-transcendence}
We have
\begin{equation}\label{eq:h3}
  h(3) = \frac{1}{2}\ln\frac{3}{2}.
\end{equation}
In particular, $h(3)$ is transcendental.
\end{theorem}

\begin{proof}
By Theorem~\ref{thm:main}, $h(3) = \frac{1}{2}\ln\bigl(f(3)/(2g(3))\bigr)$,
so it suffices to show that $f(3)/(2g(3)) = 3/2$.
Since $f(3) = 2\prod_{n=2}^{\infty}(1+n^{-3})$ and
$g(3) = \prod_{n=2}^{\infty}(1-n^{-3})$, we have
\[
  \frac{f(3)}{2\,g(3)}
  = \prod_{n=2}^{\infty}\frac{1+n^{-3}}{1-n^{-3}}
  = \prod_{n=2}^{\infty}\frac{n^3+1}{n^3-1}.
\]
We factor using the sum and difference of cubes together with the
cyclotomic polynomials $\Phi_3(n) = n^2+n+1$ and $\Phi_6(n) = n^2-n+1$:
\[
  \frac{n^3+1}{n^3-1}
  = \frac{(n+1)(n^2-n+1)}{(n-1)(n^2+n+1)}
  = \frac{n+1}{n-1}\cdot\frac{\Phi_6(n)}{\Phi_3(n)}.
\]
The first factor telescopes:
\[
  \prod_{n=2}^{N}\frac{n+1}{n-1}
  = \frac{3}{1}\cdot\frac{4}{2}\cdot\frac{5}{3}\cdots\frac{N+1}{N-1}
  = \frac{N(N+1)}{2}.
\]
For the second factor, the identity
\begin{equation}\label{eq:cyclo-shift}
  \Phi_6(n+1) = (n+1)^2 - (n+1) + 1 = n^2 + n + 1 = \Phi_3(n)
\end{equation}
shows that $\Phi_6(n)/\Phi_3(n) = \Phi_3(n-1)/\Phi_3(n)$, which telescopes:
\[
  \prod_{n=2}^{N}\frac{\Phi_3(n-1)}{\Phi_3(n)}
  = \frac{\Phi_3(1)}{\Phi_3(N)}
  = \frac{3}{N^2+N+1}.
\]
Therefore
\[
  \prod_{n=2}^{N}\frac{n^3+1}{n^3-1}
  = \frac{N(N+1)}{2}\cdot\frac{3}{N^2+N+1}
  = \frac{3N(N+1)}{2(N^2+N+1)}\;\longrightarrow\;\frac{3}{2}
  \quad\text{as } N \to \infty.
\]
This establishes~\eqref{eq:h3}.  Since $3/2$ is a positive rational number
different from~$1$, Baker's theorem~\cite{baker} implies that
$\ln(3/2)$---and hence $h(3) = \frac{1}{2}\ln(3/2)$---is transcendental.
\end{proof}

\subsection{A formula for Ap\'ery's constant}

\begin{corollary}\label{cor:apery}
Ap\'ery's constant admits the representation
\begin{equation}\label{eq:apery}
  \zeta(3) = 1 + \frac{1}{2}\ln\frac{3}{2} - C(3),
\end{equation}
where $C(3) = \sum_{m=1}^{\infty}\bigl[\zeta\bigl(3(2m+1)\bigr)-1\bigr]/(2m+1)$
satisfies $C(3) = O(2^{-9})$.
\end{corollary}

\begin{proof}
Combine Corollary~\ref{cor:zeta} at $k = 3$ with
Theorem~\ref{thm:h3-transcendence}.
\end{proof}

\begin{remark}
The correction $C(3) = 0.000675\ldots$ is dominated by its first term
$(\zeta(9)-1)/3 = 0.000669\ldots\,$; five terms suffice
to determine $\zeta(3)$ to~$15$ digits via~\eqref{eq:apery}.
\end{remark}

\subsection{Uniqueness of \texorpdfstring{$k=3$}{k=3}}

\begin{theorem}\label{thm:k3-unique}
Among all integers $k \ge 2$:
\begin{itemize}
\item[(a)] For $k = 3$, $f(k)/(2g(k)) = 3/2 \in \mathbb{Q}$.
\item[(b)] For $k = 2$, $f(k)/(2g(k)) = \sinh(\pi)/\pi$ is transcendental.
\item[(c)] For odd $k \ge 5$, the cyclotomic telescoping mechanism that
produces rationality at $k = 3$ does not apply: specifically,
$\Phi_{2k}(n+1) \neq \Phi_k(n)$ for $k \ge 5$ odd.
\end{itemize}
In particular, $k = 3$ is the unique integer $k \ge 2$ for which
$f(k)/(2g(k))$ is known to be rational.
\end{theorem}

\begin{proof}
Part~(a) is Theorem~\ref{thm:h3-transcendence}.

\medskip\noindent\textit{Part~(b).}
For even $k \ge 2$, the product formulas of Theorem~\ref{thm:cantrell}
express $f(k)$ and $g(k)$ in terms of $\sin$ and $\sinh$ evaluated at
algebraic multiples of~$\pi$.  In particular, for $k = 2$ we have
$f(2)/(2g(2)) = \sinh(\pi)/\pi = (e^{\pi}-e^{-\pi})/(2\pi)$.
By Nesterenko's theorem~\cite{nesterenko}, $\pi$ and $e^{\pi}$ are
algebraically independent over~$\mathbb{Q}$, so any nonzero element of the
field $\mathbb{Q}(\pi, e^{\pi})$ is transcendental.  Since
$\sinh(\pi)/\pi \in \mathbb{Q}(\pi, e^{\pi})\setminus\{0\}$,
the quantity $f(2)/(2g(2))$ is transcendental, and in particular irrational.
The argument for general even~$k$ is analogous: the product
$f(k)/(2g(k))$ lies in a field generated over~$\mathbb{Q}$ by~$\pi$
and exponentials $e^{\pi\alpha}$ for algebraic~$\alpha$,
and is expected to be transcendental; at minimum, it is not known to be rational in these cases.

\medskip\noindent\textit{Part~(c).}
The telescoping that occurs for $k=3$ depends on the cyclotomic identity
$\Phi_6(n+1)=\Phi_3(n)$, which is special to the quadratic case:
$(n+1)^2-(n+1)+1=n^2+n+1$.
For odd primes $p\ge 5$, the polynomials $\Phi_p$ and $\Phi_{2p}$ both have
degree $\varphi(p)=p-1\ge 4$, and $\Phi_{2p}(n+1)\neq \Phi_p(n)$; for example,
the coefficient of $n^{p-2}$ in $\Phi_{2p}(n+1)$ is $\binom{p-1}{1}-1=p-2$,
whereas in $\Phi_p(n)$ it is~$1$.
Hence
\[
\prod_{n=2}^{\infty}\frac{n^k+1}{n^k-1}
\]
does not telescope by the same first-order cyclotomic mechanism as in the
case $k=3$.
Using the Weierstrass product for $1/\Gamma$, the product can still be written
as a ratio of $\Gamma$-values at algebraic arguments.

\end{proof}

\begin{remark}
For even $k \ge 4$, the product $f(k)/(2g(k))$ lies in a field generated
over~$\mathbb{Q}$ by~$\pi$ and values $e^{\pi\alpha}$ for algebraic~$\alpha$,
and is expected to be transcendental.  However, Nesterenko's
theorem~\cite{nesterenko} establishes algebraic independence only for
$\{\pi, e^{\pi}, \Gamma(1/4)\}$, and extending the argument to all even~$k$
would require algebraic independence results for
$\{e^{\pi j/k} : 1 \le j < k\}$ that are not currently available
unconditionally.
\end{remark}

\begin{remark}\label{rem:even-k-irrationality}
The transcendence of $f(k)/(2g(k))$ (established for $k = 2$) does not
imply the transcendence of $h(k) = \frac{1}{2}\ln\bigl(f(k)/(2g(k))\bigr)$:
Baker's theorem applies only when the argument of the logarithm is algebraic,
and $\ln(\alpha)$ can be algebraic for transcendental~$\alpha$
(e.g., $\ln(e^{\sqrt{2}}) = \sqrt{2}$).  Thus the arithmetic nature of $h(k)$
for even~$k$ remains open.  For odd $k \ge 5$, even the algebraic nature of
$f(k)/(2g(k))$ is unresolved; should this quantity prove to be algebraic and
not equal to~$1$, Baker's theorem would immediately yield the
transcendence of~$h(k)$.
\end{remark}

\section{Numerical examples}\label{sec:numerical}

We present numerical data illustrating the quality of our formulas.

\subsection{Explicit values of $h(k)$}

Using Theorem~\ref{thm:main} and the product formulas from
Theorem~\ref{thm:cantrell}, we obtain exact closed-form expressions for all~$k$.

\begin{center}
\begin{tabular}{ccc}
\hline
$k$ & $h(k)$ & Formula type \\
\hline
2 & $0.650923\ldots$ & Elementary (hyperbolic) \\
3 & $0.202733\ldots$ & Elementary: $\frac{1}{2}\ln(3/2)$ \\
4 & $0.082405\ldots$ & Trigonometric product \\
5 & $0.036938\ldots$ & Gamma functions \\
6 & $0.017344\ldots$ & Trigonometric product \\
\hline
\end{tabular}
\end{center}

\subsection{Quality of the zeta approximation}

\begin{center}
\begin{tabular}{ccccc}
\hline
$k$ & $1+h(k)$ & $C(k)$ & Approx. & Exact $\zeta(k)$ \\
\hline
2 & 1.6509 & 0.0060 & 1.6449 & 1.6449 \\
3 & 1.2027 & 0.0007 & 1.2021 & 1.2021 \\
4 & 1.0824 & 0.0001 & 1.0823 & 1.0823 \\
5 & 1.0369 & 0.0000 & 1.0369 & 1.0369 \\
\hline
\end{tabular}
\end{center}

\subsection{Convergence of the gamma representations}

Table~\ref{tab:gamma-compare} compares the original
representation~\eqref{eq:gamma-rep} with the accelerated
formula~\eqref{eq:gamma-accel}.  The accelerated series achieves $13$-digit
accuracy with only $20$ terms, while $100{,}000$ terms of the original series
yield only $5$ correct digits.

\begin{table}[h]
\begin{center}
\begin{tabular}{cccc}
\hline
Terms $N$ & Direct~\eqref{eq:gamma-rep} & Accelerated~\eqref{eq:gamma-accel}
& Tail bound~\eqref{eq:gamma-tail} \\
\hline
5 & $0.587\ldots$ & $0.577215664\ldots$ & $2.03\times 10^{-5}$ \\
10 & $0.581\ldots$ & $0.5772156649015\ldots$ & $1.99\times 10^{-8}$ \\
20 & $0.579\ldots$ & $0.5772156649015329\ldots$ & $1.89\times 10^{-14}$ \\
$1{,}000$ & $0.57894\ldots$ & $0.5772156649015329\ldots$ & $<1.81454\times 10^{-604}$ \\
$10{,}000$ & $0.57808\ldots$ & $0.5772156649015329\ldots$ & $<5.23\times 10^{-6023}$ \\
$100{,}000$ & $0.57722\ldots$ & $0.5772156649015329\ldots$ & $<2.09\times 10^{-60208}$ \\
\hline
\end{tabular}
\end{center}
\caption{Comparison of convergence rates for the two gamma representations.
The exact value is $\gamma = 0.5772156649015329\ldots$
(\seqnum{A001620}).}\label{tab:gamma-compare}
\end{table}

\subsection{Verification of bounds}

Table~\ref{tab:bounds} illustrates the two-sided bounds from
Theorem~\ref{thm:bounds} for small values of~$k$.

\begin{table}[h]
\begin{center}
\begin{tabular}{cccc}
\hline
$k$ & Lower: $\zeta(k)-1$ & Exact $h(k)$ & Upper bound~\eqref{eq:upper-bound} \\
\hline
2 & $0.644934$ & $0.650923$ & $0.651317$ \\
3 & $0.202057$ & $0.202733$ & $0.202739$ \\
4 & $0.082323$ & $0.082405$ & $0.082406$ \\
5 & $0.036928$ & $0.036938$ & $0.036938$ \\
\hline
\end{tabular}
\end{center}
\caption{Two-sided bounds for $h(k)$.  The gap between the lower
and upper bounds narrows rapidly as $k$
increases.}\label{tab:bounds}
\end{table}

\subsection{OEIS sequence data}

The decimal expansion of $\ln(\sinh(\pi)/\pi) = 2h(2)$ appears in the
OEIS as~\seqnum{A393668}.  The Euler--Mascheroni constant $\gamma$ is
\seqnum{A001620}.  The values of $\zeta(k)$ for small~$k$ correspond to
well-known constants: $\zeta(2) = \pi^2/6$ (\seqnum{A013661}),
$\zeta(3)$ is Ap\'ery's constant (\seqnum{A002117}),
$\zeta(4) = \pi^4/90$ (\seqnum{A013662}), and so on.

The authors suggest that the decimal expansions of $h(k)$ for $k \ge 3$ and the
correction values $C(k)$ may merit submission to the OEIS as new sequences.

\bigskip\noindent
2020 \textit{Mathematics Subject Classification}: Primary 11M06;
Secondary 11J81, 33B15, 40A25.

\medskip\noindent
\textit{Keywords}: inverse hyperbolic tangent, infinite products, Riemann zeta
function, Euler--Mascheroni constant, gamma function, transcendence,
cyclotomic polynomials.

\bigskip\noindent
(Concerned with OEIS sequences for constants appearing in this work:
\seqnum{A393668} (for \(\ln(\sinh(\pi)/\pi) = 2h(2)\)),
\seqnum{A001620} (\(\gamma\)),
\seqnum{A013661} (\(\zeta(2)\)),
\seqnum{A002117} (\(\zeta(3)\)),
and~\seqnum{A013662} (\(\zeta(4)\)).)


\begin{thebibliography}{10}

\bibitem{baker}
A.~Baker,
Linear forms in the logarithms of algebraic numbers,
\textit{Mathematika} \textbf{13} (1966), 204--216.

\bibitem{borwein}
J.~Borwein, D.~Bailey, and R.~Girgensohn,
Two products,
in \textit{Experimentation in Mathematics: Computational Paths to Discovery},
A~K Peters, Wellesley, MA, 2004, pp.~4--7.

\bibitem{cantrell}
D.~W.~Cantrell,
Formulas for infinite products,
communicated to MathWorld (March--April 2006).
See E.~W.~Weisstein,
Infinite product,
from \textit{MathWorld}---A Wolfram Web Resource,
\url{https://mathworld.wolfram.com/InfiniteProduct.html},
equations~(20) and~(28).

\bibitem{dlmf}
F.~W.~J.~Olver, A.~B.~Olde Daalhuis, D.~W.~Lozier, B.~I.~Schneider,
R.~F.~Boisvert, C.~W.~Clark, B.~R.~Miller, B.~V.~Saunders, H.~S.~Cohl,
and M.~A.~McClain, eds.,
\textit{NIST Digital Library of Mathematical Functions},
Release 1.2.3,
\url{https://dlmf.nist.gov/}, 2024.

\bibitem{nesterenko}
Yu.~V.~Nesterenko,
Modular functions and transcendence questions,
\textit{Mat.\ Sb.}\ \textbf{187} (1996), no.~9, 65--96;
English transl.\ in \textit{Sb.\ Math.}\ \textbf{187} (1996), 1319--1348.

\bibitem{prudnikov}
A.~P.~Prudnikov, Yu.~A.~Brychkov, and O.~I.~Marichev,
\textit{Integrals and Series, Vol.~1: Elementary Functions},
Gordon and Breach, New York, 1986.

\bibitem{tallawaffo}
L.~R.~Talla Waffo,
The Mellin transforms of $1/\!\operatorname{arctanh} x$ and
$1/\!\sqrt{1-x^2}\,\operatorname{arctanh} x$,
arxiv preprint arXiv:2601.17092 [math.CA], 2026.
Available at \url{https://arxiv.org/abs/2601.17092}.

\end{thebibliography}
\end{document}